\documentclass[12pt]{amsart}

\usepackage[utf8]{inputenc}
\usepackage[T1]{fontenc}
\usepackage[english]{babel}

\usepackage{amsfonts}
\usepackage{amsmath}
\usepackage{amssymb}
\usepackage{amstext}
\usepackage{amsthm}

\usepackage{geometry}
\usepackage{fullpage}

\usepackage[shortlabels]{enumitem}
\usepackage{graphicx}
\usepackage{hyperref}
\usepackage{mathrsfs}
\usepackage{mathtools}

\allowdisplaybreaks

\newcommand{\wtF}{\widetilde{F}}
\newcommand{\wtH}{\widetilde{H}}
\newcommand{\wtm}{\widetilde{m}}
\newcommand{\N}{\mathbb{N}}
\newcommand{\Z}{\mathbb{Z}}
\newcommand{\Q}{\mathbb{Q}}
\DeclareMathOperator{\Cyc}{Cyc}

\newtheorem{theorem}{Theorem}[section]
\newtheorem{lemma}[theorem]{Lemma}
\newtheorem{proposition}[theorem]{Proposition}

\newtheorem{conjecture}[theorem]{Conjecture}

\theoremstyle{definition}

\newtheorem{example}[theorem]{Example}

\theoremstyle{remark}
\newtheorem{remark}[theorem]{Remark}

\theoremstyle{plain}
\newtheorem*{claim}{Claim}

\title{The Schur positivity of $\nabla m_\mu$}

\author{Dun Qiu}
\address{Center for Combinatorics, LPMC, Nankai University, Tianjin 300071, P. R. China}
\email{qiudun@nankai.edu.cn}
\thanks{Dun Qiu is supported in part by the Fundamental Research Funds for the Central Universities (105-63263093), the National Natural Science Foundation of China (12271023 and 12171034),
	and the Natural Science Foundation of Tianjin (24JCZDJC01390). We thank Jaeseong Oh for his valuable comments on an initial version of this paper.}
\author{Minhao Zhang}
\address{Center for Combinatorics, LPMC, Nankai University, Tianjin 300071, P. R. China}
\email{carpentercheungmath@gmail.com}

\date{\today}



\begin{document}

\begin{abstract}
	Bergeron, Garsia, Haiman and Tesler conjectured in 1999 that, for all partitions
	$\mu,\lambda\vdash n$, the polynomial
	$(-1)^{|\mu|-\ell(\mu)}\langle \nabla m_\mu, s_\lambda\rangle$ has nonnegative
	integer coefficients, where $\nabla$ is the Bergeron--Garsia nabla operator,
	which acts diagonally on the modified Macdonald basis, and
	$m_\mu$ is the monomial symmetric function. In this article, we prove this conjecture, and more
	generally that
	$(-1)^{|\mu|-\ell(\mu)}\langle\nabla^r m_\mu,s_\lambda\rangle\in\N[q,t]$
	for all $r\geq 1$. 
	We establish a recursion showing that $(-1)^{|\mu|-\ell(\mu)}m_\mu$
	has an expansion with coefficients in $\Q_{\geq 0}[q]$
	in the symmetric functions $C_\alpha(1)$, where $C_a$ denotes the
	operator introduced by Haglund, Morse and Zabrocki.
	Combining this
	expansion with the compositional shuffle theorems of Carlsson--Mellit and
	Mellit, and with the Schur positivity of LLT polynomials, completes the proof.
	The same method, using the $e$-positivity of column LLT polynomials after the
	substitution $q\mapsto q+1$, also gives an $e$-positive analogue.
\end{abstract}

\maketitle

\section{Introduction}
In 1988, Macdonald introduced a two-parameter basis of the ring of symmetric
functions, now known as the Macdonald polynomials \cite{Mac88}. Garsia and
Haiman introduced the modified Macdonald polynomials
$\wtH_\mu[X;q,t]$ and constructed the Garsia--Haiman modules, a
family of bigraded $S_n$-modules, conjecturing that their Frobenius
characteristics are the modified Macdonald polynomials
\cite{GH93,GH96b}. This conjectural representation-theoretic interpretation
was proved by Haiman through his work on Hilbert schemes and the
$n!$-theorem \cite{Hai01}.

Garsia and Haiman also introduced the space of diagonal harmonics and
conjectured that its bigraded Frobenius characteristic is
$\nabla e_n$ \cite{GH96}, where $\nabla$ is the Bergeron--Garsia
nabla operator introduced in \cite{BG99}.
Haiman proved the diagonal harmonics conjecture in \cite{Hai02}.
Consequently,
\[
\langle \nabla e_n, s_\lambda\rangle \in \N[q,t]
\]
for every partition $\lambda\vdash n$.

In 2005, Haglund, Haiman, Loehr, Remmel and Ulyanov conjectured a positive
combinatorial LLT polynomial expansion for $\nabla e_n$ \cite{HHLRU05}.
Haglund, Morse and Zabrocki introduced the operators $C_a$ and formulated a
compositional refinement of this conjecture \cite{HMZ12}, the compositional
shuffle theorem, proved by Carlsson and Mellit \cite{CM18}. Gorsky and Negut
\cite{GorskyNegut15} and Bergeron, Garsia, Leven and Xin \cite{BGLX16}
generalized it to the rational $(km,kn)$-setting, proved by Mellit
\cite{Mel21}. For recent developments, see
\cite{BlasiakHaimanMorsePunSeelinger24,BlasiakHaimanMorsePunSeelinger25b,
BlasiakHaimanMorsePunSeelinger23b,KimOh26,KimLeeOh25,KimOh24}. Combining the
resulting compositional $(km,kn)$-shuffle theorem with the Schur positivity of LLT
polynomials proved by Grojnowski and Haiman \cite{GrojnowskiHaiman07}, one
obtains
\[
\langle \nabla^r C_\alpha(1), s_\lambda\rangle \in \N[q,t]
\]
for every composition $\alpha$, partition $\lambda$, and integer
$r\ge 1$. Since we could not locate this exact statement for $r\geq 2$
in the literature, we give a proof in Proposition~\ref{propnablaC}.

The present paper concerns the signed Schur positivity of $\nabla m_\mu$.
Bergeron, Garsia, Haiman and Tesler made the following conjecture.

\begin{conjecture}[\cite{BGHT99}, Conjecture IV]\label{conj:BGHT-IV}
	For any partitions $\mu,\lambda\vdash n$,
	\[
	(-1)^{|\mu|-\ell(\mu)}
	\left\langle \nabla m_\mu, s_\lambda \right\rangle
	\in \N[q,t].
	\]
\end{conjecture}

When $\mu=(1^n)$, we have $m_{(1^n)}=e_n$, so this case follows from
Haiman's theorem on diagonal harmonics \cite{Hai02}. The case
$\mu=(n)$, where $m_{(n)}=p_n$, is related to the square paths
theorem, conjectured by Loehr and Warrington \cite{LW07} and proved by
Sergel \cite{Ser17}. Sergel also proved the cases where $\mu$ is a hook, and
conjectured a more general combinatorial model for $\nabla m_\mu$ in
\cite{Ser18}. More recently, Qu and Xin proved the cases
$\mu=(2^k,1^\ell)$ in \cite{QX25}, by establishing a recursion which
implies that $(-1)^{k}m_{(2^k,1^\ell)}$ has a positive expansion in the
symmetric functions $C_\alpha(1)$. Recursive arguments of this type originate
from the calculus of the operators $C_a$ introduced by Haglund, Morse
and Zabrocki \cite{HMZ12}. This is the same calculus that underlies the
compositional shuffle theorem of Carlsson and Mellit \cite{CM18} and its
$(km,kn)$-extension, conjectured by Bergeron, Garsia, Leven and
Xin \cite{BGLX16} and proved by Mellit \cite{Mel21}. The present paper extends this
recursive approach to arbitrary partitions; all operator identities we need
are derived from the definition of the $C_a$ and the basic identities of
Haglund, Morse and Zabrocki \cite{HMZ12} recalled in Section~2.

Bergeron, Garsia, Haiman and Tesler also proved the following integrality
statement in \cite{BGHT99}:
\[
(-1)^{|\mu|-\ell(\mu)}
\left\langle \nabla m_\mu, s_\lambda \right\rangle
\in \Z[q,t].
\]
In \cite{BlasiakHaimanMorsePunSeelinger25}, Blasiak, Haiman, Morse, Pun and
Seelinger proved the Loehr--Warrington conjecture \cite{LW08} by giving an
LLT expansion of $\nabla^r s_\lambda$ with integral coefficients.
Combining this with the Schur positivity of LLT polynomials due to
Grojnowski and Haiman \cite{GrojnowskiHaiman07} and the integrality of the
inverse Kostka matrix established by E\u{g}ecio\u{g}lu and Remmel
\cite{ER90}, one obtains
\begin{equation}\label{nablamZ}
	(-1)^{|\mu|-\ell(\mu)}
	\left\langle \nabla^r m_\mu, s_\lambda \right\rangle
	\in \Z[q,t]
	\qquad\text{for all } r\geq 1.
\end{equation}

The main result of this paper is the following positivity theorem.

\begin{theorem}\label{mainthm}
	For any partitions $\mu,\lambda\vdash n$ and any integer $r\ge 1$,
	\[
	(-1)^{|\mu|-\ell(\mu)}
	\langle \nabla^r m_\mu, s_\lambda\rangle \in \N[q,t].
	\]
\end{theorem}

In particular, Theorem~\ref{mainthm} resolves Conjecture~\ref{conj:BGHT-IV}
in full generality.
For a labelled multiset $A$ with $\mu(A)=\mu$, the symmetric function
$\wtF_A$ is a positive integer multiple of $(-1)^{|\mu|-\ell(\mu)}m_\mu$.
We prove that it admits a $C$ expansion, an expansion in the symmetric
functions $C_\alpha(1)$, with coefficients in $\N[q]$ (Theorem~\ref{thm2});
this expansion is produced by a new recursion for the $\wtF_A$
(Theorem~\ref{thm1}). By the Schur positivity of $\nabla^r C_\alpha(1)$
(Proposition~\ref{propnablaC}), it follows that
$(-1)^{|\mu|-\ell(\mu)}\langle \nabla^r m_\mu, s_\lambda\rangle$ has
nonnegative rational coefficients. Combined with the integrality
statement \eqref{nablamZ}, this yields Theorem~\ref{mainthm}.

The same method yields an $e$-positive analogue.  By the integrality
\eqref{nablamZ}, together with the integral transition between the Schur and
elementary symmetric functions and the fact that the substitution
$q\mapsto q+1$ preserves $\Z[q,t]$, we have
\begin{equation}\label{nablamZe}
	(-1)^{|\mu|-\ell(\mu)}
	\left\langle (\nabla^r m_\mu)[X;q+1], \omega(m_\lambda) \right\rangle
	\in \Z[q,t]
	\qquad\text{for all } r\geq 1 .
\end{equation}
Here and below, for $f\in\Lambda$, the notation $f[X;q+1]$ means that the
parameter $q$ in $f$ is replaced by $q+1$.  An explicit $e$-positive
expansion of column LLT polynomials after this substitution was conjectured
independently by Alexandersson \cite{Alexandersson20} and by Garsia,
Haglund, Qiu and Romero \cite{GHQR19}; the $e$-positivity was proved by
D'Adderio \cite{DAdderio19}, and the conjectured expansion was established by
Alexandersson and Sulzgruber \cite{AlexanderssonSulzgruber22}.  Since $\nabla^r C_\alpha(1)$ is a nonnegative integer combination of
column LLT polynomials by Mellit's compositional $(km,kn)$-shuffle theorem \cite{Mel21},
D'Adderio's $e$-positivity yields the following companion to
Theorem~\ref{mainthm}.

\begin{theorem}\label{maine}
	For any partitions $\mu,\lambda\vdash n$ and any integer $r\ge 1$,
	\[
	(-1)^{|\mu|-\ell(\mu)}
	\left\langle (\nabla^r m_\mu)[X;q+1], \omega(m_\lambda)\right\rangle
	\in \N[q,t].
	\]
\end{theorem}

The paper is organized as follows. In Section~2, we recall the
$C$ operators and their basic properties, and establish two positivity
properties of $\nabla^r C_\alpha(1)$: Schur positivity, and $e$-positivity
after the substitution $q\mapsto q+1$ (Proposition~\ref{propnablaC}). In Section~3,
we introduce labelled multisets and the symmetric functions $\widetilde
F_A$, and establish several auxiliary lemmas. In Section~4, we establish
a recursion for $\wtF_A$ (Theorem~\ref{thm1}), derive from it a
$C$ expansion with coefficients in $\N[q]$ (Theorem~\ref{thm2}),
and use this expansion to prove Theorems~\ref{mainthm} and~\ref{maine}.

\section{Background}
In this section, we recall the $C$ operators introduced by Haglund, Morse
and Zabrocki \cite{HMZ12}, and fix notation used throughout the paper. We
follow the notation of Macdonald \cite{Mac95} and Haglund \cite{Hag08}.

\subsection{Symmetric functions}
A partition is a sequence of positive integers
$\lambda=(\lambda_1,\ldots,\lambda_\ell)$ with
$\lambda_1\geq \lambda_2\geq \cdots \geq \lambda_\ell$.
If $|\lambda|=\lambda_1+\cdots+\lambda_\ell=n$, we write
$\lambda\vdash n$. We denote by $\ell(\lambda)$ the length of
$\lambda$, and by $m_i(\lambda)$ the multiplicity of $i$ among the
parts of $\lambda$. A composition of $n$ is a sequence
$\alpha=(\alpha_1,\ldots,\alpha_m)$ of positive integers with
$\alpha_1+\cdots+\alpha_m=n$; in this case we write $\alpha\vDash n$.
We regard the empty sequence $\emptyset$ as the unique composition of
$0$.

Let $\Lambda=\bigoplus_{n\geq 0}\Lambda^n$ be the ring of symmetric
functions over $\Q(q,t)$, where $\Lambda^n$ denotes the space
of homogeneous symmetric functions of degree $n$. We use the standard
bases of $\Lambda$: the monomial, elementary, complete homogeneous,
power-sum, and Schur symmetric functions, denoted $m_\mu, e_\mu, h_\mu,
p_\mu$, and $s_\mu$, respectively.

The involutory automorphism $\omega: \Lambda \to \Lambda$ is defined by
$\omega(e_n) = h_n$. The bilinear inner product, called the \emph{Hall inner
product}, is defined by
\[
\langle s_\lambda, s_\mu\rangle = \delta_{\lambda,\mu}.
\]
By the duality properties of the Hall inner product, the coefficient of
$s_\lambda$ in a symmetric function $f$ is $\langle f,s_\lambda\rangle$,
whereas the coefficient of $e_\lambda$ in $f$ is
$\langle f, \omega(m_\lambda) \rangle$.

For convenience, we write $F_\mu := (-1)^{|\mu|-\ell(\mu)}m_\mu$; when
$\mu=(n)$ consists of a single part, we abbreviate
$F_n:=F_{(n)}=(-1)^{n-1}p_n=\omega(p_n)$. (The 
letter $F$ is unrelated to Gessel's fundamental quasisymmetric functions,
which are not assigned a symbol in this paper.)
We use the Newton identity
\[
n h_n=\sum_{i=1}^n p_i h_{n-i}.
\]

Let $\wtH_\mu[X;q,t]$ denote the modified Macdonald polynomial.
The nabla operator $\nabla$ is defined as the diagonal operator on the
modified Macdonald basis by
\[
\nabla \wtH_\mu[X;q,t] = T_\mu\,\wtH_\mu[X;q,t],
\qquad\text{where}\qquad
T_\mu=\prod_{c\in\mu} q^{a'(c)}t^{\ell'(c)}.
\]
Here $a'(c)$ and $\ell'(c)$ denote the coarm and coleg of the cell
$c\in\mu$, respectively; see Haglund \cite{Hag08}.

We use plethystic notation throughout the paper. Let
\(X=x_1+x_2+\cdots\). For a symmetric function $f\in\Lambda$ and an
expression $A$ in some variables and parameters, the plethystic evaluation
$f[A]$ is defined by first writing $f$ as a polynomial in the power-sum
functions $p_k$, and then replacing each $p_k$ by $p_k[A]$, where
$p_k[A]$ is obtained from $A$ by raising every monomial in $A$
to the $k$-th power.

A symmetric function $f\in\Lambda$ is called \emph{Schur positive} if
\[
f=\sum_\mu c_\mu(q,t)s_\mu
\qquad\text{with}\qquad
c_\mu(q,t)\in\N[q,t],
\]
equivalently, if $\langle f,s_\mu\rangle\in\N[q,t]$ for every $\mu$.
Similarly, a symmetric function $f\in\Lambda$ is called \emph{$e$-positive} if
\[
f=\sum_\mu c_\mu(q,t)e_\mu
\qquad\text{with}\qquad
c_\mu(q,t)\in\N[q,t],
\]
equivalently, if $\langle f,\omega(m_\mu)\rangle\in\N[q,t]$ for every $\mu$.
By convention, $[n]_q=1+q+\cdots+q^{n-1}$ denotes the $q$-integer.

\subsection{The \texorpdfstring{$C$}{C} operators}
In \cite{HMZ12}, Haglund, Morse and Zabrocki introduced the operators
$C_a$, defined by
\begin{equation}\label{eqdefC}
	C_a(f)[X]
	=
	\left(-\frac{1}{q}\right)^{a-1}
	\left\langle z^a \right\rangle
	f\left[X-\frac{1-1/q}{z}\right]
	\sum_{r\geq 0} h_r[X]z^r,
\end{equation}
where $\langle z^a\rangle G(z)$ denotes the coefficient of $z^a$ in
the formal Laurent series $G(z)$.

For any symmetric function $f$ and any composition
$\alpha=(\alpha_1,\ldots,\alpha_m)$, we write
\begin{equation}
	C_\alpha(f) = C_{\alpha_1}\bigl(C_{\alpha_2}(\cdots C_{\alpha_m}(f)\cdots)\bigr).
\end{equation}
For the empty composition we use the convention $C_\emptyset(f)=f$, so
$C_\emptyset(1)=1$.  Note that $C_a$ sends $\Lambda^n$ to
$\Lambda^{n+a}$, so that $C_\alpha(1)\in\Lambda^{|\alpha|}$. We also
recall the following identity of Haglund, Morse and Zabrocki \cite{HMZ12}:
\begin{equation}\label{eqHMZen}
	e_n=\sum_{\alpha\vDash n}C_\alpha(1).
\end{equation}

The following two lemmas, which we use repeatedly in Sections~3
and~4, are derived from the definition \eqref{eqdefC}.

\begin{lemma}\label{lemma3}
	For every positive integer $n$,
	\begin{equation}\label{eqlemma3}
		F_n
		=
		\sum_{j=1}^{n-1} q^{j-1} C_j(F_{n-j})
		+
		[n]_q C_n(1).
	\end{equation}
\end{lemma}

\begin{proof}
	Fix $1\leq j\leq n-1$, and set $m=n-j$. By the definition
	\eqref{eqdefC} of $C_j$,
	\begin{align*}
		q^{j-1} C_j(F_m)
		&=
		q^{j-1}
		\left(-\frac{1}{q}\right)^{j-1}
		\left\langle z^j\right\rangle
		F_m\left[X-\frac{1-1/q}{z}\right]
		\sum_{r\geq 0}h_r[X]z^r \\
		&=
		(-1)^{j-1}
		\left\langle z^j\right\rangle
		(-1)^{m-1}
		p_m\left[X-\frac{1-1/q}{z}\right]
		\sum_{r\geq 0}h_r[X]z^r \\
		&=
		(-1)^{m+j-2}
		\left\langle z^j\right\rangle
		\left(
		p_m[X]
		-
		\frac{1-q^{-m}}{z^m}
		\right)
		\sum_{r\geq 0}h_r[X]z^r \\
		&=
		(-1)^n
		\left(
		p_m[X]h_j[X]
		-
		(1-q^{-m})h_{m+j}[X]
		\right) \\
		&=
		(-1)^n p_m[X]h_j[X]
		+
		(-1)^{n-1}(1-q^{-m})h_n[X].
	\end{align*}
	Since $C_n(1)=(-1/q)^{\,n-1}h_n[X]$ directly from \eqref{eqdefC},
	substituting $m=n-j$ and summing over $j=1,\ldots,n-1$ gives
	\begin{align*}
		&\quad\quad \sum_{j=1}^{n-1}q^{j-1}C_j(F_{n-j})
		+
		[n]_q C_n(1) \\
		&=
		(-1)^n\sum_{j=1}^{n-1}p_{n-j}[X]h_j[X]
		+
		(-1)^{n-1}
		\sum_{j=1}^{n-1}
		\left(1-q^{-(n-j)}\right)h_n[X]
		+
		(-1)^{n-1}
		\left(\sum_{m=0}^{n-1}q^{-m}\right)h_n[X] \\
		&=
		(-1)^n\sum_{j=1}^{n-1}p_{n-j}[X]h_j[X]
		+
		(-1)^{n-1}h_n[X]
		\left(
		\sum_{m=1}^{n-1}(1-q^{-m})
		+
		\sum_{m=0}^{n-1}q^{-m}
		\right) \\
		&=
		(-1)^n\sum_{j=1}^{n-1}p_{n-j}[X]h_j[X]
		+
		(-1)^{n-1}n h_n[X].
	\end{align*}
	By the Newton identity,
	\[
	p_n[X]
	=
	n h_n[X]
	-
	\sum_{i=1}^{n-1}p_i[X]h_{n-i}[X].
	\]
	Reindexing via $i=n-j$ turns the Newton identity into
	\[
	(-1)^{n-1}p_n[X]=(-1)^{n}\sum_{j=1}^{n-1}p_{n-j}[X]h_j[X]+(-1)^{n-1}n h_n[X],
	\]
	so
	\[
	\sum_{j=1}^{n-1}q^{j-1}C_j(F_{n-j})
	+
	[n]_q C_n(1)
	=
	(-1)^{n-1}p_n[X]
	=
	F_n. \qedhere
	\]
\end{proof}

\begin{lemma}\label{lemma4}
	For all positive integers $a$ and $i$, and for any symmetric
	function $G$,
	\begin{equation}\label{eqlemma4}
		F_a C_i(G) = C_i(F_a G) + (1-q^a)C_{a+i}(G).
	\end{equation}
\end{lemma}

\begin{proof}
	By the definition \eqref{eqdefC} of $C_i$,
	\begin{align*}
		F_a C_i(G) - C_i(F_a G)
		&=
		\left(-\frac{1}{q}\right)^{i-1}
		\left\langle z^i\right\rangle
		\left(
		F_a[X]
		-
		F_a\left[X-\frac{1-1/q}{z}\right]
		\right) \\
		&\qquad
		G\left[X-\frac{1-1/q}{z}\right]
		\sum_{r\geq 0}h_r[X]z^r.
	\end{align*}
	Since $F_a=(-1)^{a-1}p_a$ and
	$p_a\bigl[\frac{1-1/q}{z}\bigr] = (1-q^{-a})z^{-a}$,
	\[
	F_a[X]
	-
	F_a\left[X-\frac{1-1/q}{z}\right]
	=
	(-1)^{a-1}(1-q^{-a})z^{-a}.
	\]
	Therefore,
	\begin{align*}
		F_a C_i(G) - C_i(F_a G)
		&=
		\left(-\frac{1}{q}\right)^{i-1}
		(-1)^{a-1}(1-q^{-a})
		\left\langle z^{i+a}\right\rangle
		G\left[X-\frac{1-1/q}{z}\right]
		\sum_{r\geq 0}h_r[X]z^r \\
		&=
		(1-q^a)
		\left(-\frac{1}{q}\right)^{i+a-1}
		\left\langle z^{i+a}\right\rangle
		G\left[X-\frac{1-1/q}{z}\right]
		\sum_{r\geq 0}h_r[X]z^r \\
		&=
		(1-q^a)C_{a+i}(G). \qedhere
	\end{align*}
\end{proof}

\subsection{Schur positivity of \texorpdfstring{$\nabla^r C_\alpha(1)$}{nabla\^{}r C\_alpha(1)}}
The following result is used in the proof of Theorem~\ref{mainthm}. For
$r=1$ it is an immediate consequence of the compositional shuffle theorem
of Carlsson and Mellit \cite{CM18} and the Schur positivity of LLT
polynomials proved by Grojnowski and Haiman \cite{GrojnowskiHaiman07}. For
general $r\geq 1$ it follows from Mellit's compositional
$(km,kn)$-shuffle theorem \cite{Mel21}. As this
exact formulation does not appear explicitly in the literature, we include
a short derivation for completeness.

\begin{proposition}\label{propnablaC}
	For every integer $r\geq 1$, every composition $\alpha\vDash n$, and
	every partition $\lambda\vdash n$,
	\[
	\langle \nabla^r C_\alpha(1), s_\lambda\rangle \in \N[q,t]
	\quad\text{and}\quad
	\langle (\nabla^r C_\alpha(1))[X;q+1], \omega(m_\lambda)\rangle \in \N[q,t].
	\]
\end{proposition}

\begin{proof}
	Mellit's compositional $(km,kn)$-shuffle theorem \cite{Mel21} states that
	for every coprime pair $(m,n_0)$ of positive integers, every integer
	$k\geq 1$, and every composition $\beta\vDash k$, the symmetric function
	$(-1)^{k(n_0+1)}\,C^{(\beta)}_{km,kn_0}\cdot 1$ is a sum of column LLT
	polynomials with coefficients in $\N[q,t]$; here
	$C^{(\beta)}_{km,kn_0}$ is the operator constructed from $C_\beta(1)$ by
	Bergeron, Garsia, Leven and Xin
	\cite[Algorithm~3.1 and Conjecture~3.3]{BGLX16}, and the LLT expansion of
	the combinatorial side is given in \cite[eqs.~(6.2)--(6.3)]{BGLX16}.
	
	We specialize to $(m,n_0)=(r,1)$, $k=n$ and $\beta=\alpha$, so that
	the sign $(-1)^{k(n_0+1)}=(-1)^{2n}$ equals $1$. By Theorem~5.1,
	Remark~5.1 and Proposition~5.1 of Bergeron, Garsia, Leven and Xin
	\cite{BGLX16}, we have
	\[
	C^{(\alpha)}_{rn,n}\cdot 1
	=\nabla^{r-1}\bigl(\nabla\,\underline{C_\alpha(1)}\,
	\nabla^{-1}\bigr)\nabla^{-(r-1)}\cdot 1
	=\nabla^{r}\,\underline{C_\alpha(1)}\,\nabla^{-r}\cdot 1
	=\nabla^{r}C_\alpha(1),
	\]
	where $\underline{C_\alpha(1)}$ denotes the operator of multiplication by
	$C_\alpha(1)$, and the last equality uses that $\nabla^{-r}$ fixes
	constants. Consequently, $\nabla^{r}C_\alpha(1)$ is a sum of column LLT
	polynomials with coefficients in $\N[q,t]$. Since every LLT
	polynomial is Schur positive by Grojnowski and Haiman
	\cite{GrojnowskiHaiman07}, it follows that $\nabla^{r}C_\alpha(1)$ is
	Schur positive.
	By the $e$-positivity of column LLT polynomials after the substitution
	$q\mapsto q+1$, proved by D'Adderio \cite{DAdderio19}, the symmetric
	function
	$(\nabla^r C_\alpha(1))[X;q+1]$ is $e$-positive.
\end{proof}

\section{Labelled multisets}
In this section, we introduce labelled multisets and the symmetric functions
$\wtF_A$, and establish the auxiliary lemmas needed in
Section~4.

A \emph{labelled multiset} $A$ is a finite multiset of positive integers
in which equal values are distinguished by labels: when a value occurs more
than once, we label its copies by subscripts $1,2,\ldots$, and when a value
occurs only once, we write it without a subscript. For
instance, $A=\{1,2_1,2_2\}$ has three pairwise distinct elements, one of
value $1$ and two of value $2$. Because its elements are pairwise
distinct, a labelled multiset may be treated as a finite set: subsets,
unions, differences and set partitions of labelled multisets are taken in
the ordinary sense.

Whenever an element of a labelled multiset appears in a formula, it stands
for its integer value; for instance, $2_1$ and $2_2$ both stand for
the integer $2$ in sums and inequalities. With this convention, we set
\[
S_A:=\sum_{x\in A}x,
\qquad
\ell(A):=|A|,
\]
so that $S_A$ is the sum of the values of the elements of $A$ (with $S_\emptyset=0$), and
$\ell(A)$ is the number of elements of $A$.

Let $\mu(A)$ be the partition whose parts are the values of the elements
of $A$, arranged in weakly decreasing order. We define
\[
    F_A := F_{\mu(A)}
    =
    (-1)^{|\mu(A)|-\ell(\mu(A))}m_{\mu(A)}.
\]
Since $|\mu(A)|=S_A$ and $\ell(\mu(A))=\ell(A)$, we have
$F_A=(-1)^{S_A-\ell(A)}m_{\mu(A)}$. In particular, for a singleton
$A=\{n\}$ we have $F_{\{n\}}=F_{(n)}=F_n=(-1)^{n-1}p_n$.

We further define
\[
    \wtF_A
    :=
    \left(\prod_{i\geq 1} m_i(A)!\right)F_A,
\]
where $m_i(A)$ denotes the multiplicity of $i$ in $A$, so that
$m_i(A)=m_i(\mu(A))$. Subsets and set partitions of $A$ treat
elements with equal values as distinct, while the symmetric functions
$F_A$ and $\wtF_A$ depend only on the partition $\mu(A)$.
Accordingly, for a partition $\mu$ we also write $\wtF_\mu:=\wtF_A$ for
any labelled multiset $A$ with $\mu(A)=\mu$; explicitly,
$\wtF_\mu=\bigl(\prod_{i\geq 1}m_i(\mu)!\bigr)F_\mu
=(-1)^{|\mu|-\ell(\mu)}\,\wtm_\mu$, where
$\wtm_\mu:=\bigl(\prod_{i\geq 1}m_i(\mu)!\bigr)m_\mu$ is the
augmented monomial symmetric function, which appears for instance in
Stanley's work on chromatic symmetric functions \cite{Sta95}.  We use the
convention $F_\emptyset=\wtF_\emptyset=1$.

It will also be useful to keep the labelled version of the augmented monomial
function visible.  For a labelled multiset $A=\{a_1,\ldots,a_k\}$, set
\[
    M_A
    :=
    \sum_{\substack{i_1,i_2,\ldots,i_k\\ \mathrm{distinct}}}
    x_{i_1}^{a_1}x_{i_2}^{a_2}\cdots x_{i_k}^{a_k},
\]
with $M_\emptyset=1$.  Then $M_A=\wtm_{\mu(A)}$, and hence
\[
    \wtF_A=(-1)^{S_A-\ell(A)}M_A.
\]

We illustrate these definitions with an example.

\begin{example}
    If $A=\{1,2_1,2_2,3_1,3_2,3_3\}$, then
    \[
        S_A=14,\qquad \ell(A)=6,\qquad
        \mu(A)=(3,3,3,2,2,1),
    \]
    so that
    \[
        F_A
        =
        (-1)^{14-6}m_{(3,3,3,2,2,1)}
        =
        m_{(3,3,3,2,2,1)}
        \qquad\text{and}\qquad
        \wtF_A
        =
        3!\,2!\,1!\,F_A
        =
        12\,m_{(3,3,3,2,2,1)}.
    \]
\end{example}

We begin with several auxiliary lemmas.  The first is the augmented,
labelled form of the E\u{g}ecio\u{g}lu--Remmel brick-tabloid formula for the
transition from the monomial basis to the power-sum basis \cite{ER91}.  We
recall the form of the formula that we need, using the notation of
E\u{g}ecio\u{g}lu and Remmel.  For partitions $\lambda,\mu\vdash n$, let
$B_{\lambda,\mu}$ be the set of $\lambda$-brick tabloids of shape
$\mu$, that is, tilings of the Young diagram of shape $\mu$ by horizontal
bricks whose lengths are the parts of $\lambda$, with bricks of the same
length regarded as indistinguishable.  If $\tau\in B_{\lambda,\mu}$, let
$\operatorname{wt}(\tau)$ be the product of the lengths of the rightmost bricks
in the rows of $\tau$, and set
$\operatorname{wt}(B_{\lambda,\mu})=\sum_{\tau\in B_{\lambda,\mu}}\operatorname{wt}(\tau)$.  Then
E\u{g}ecio\u{g}lu and Remmel's formula gives
\[
    m_\lambda
    =
    \sum_{\mu\vdash n}
    (-1)^{\ell(\lambda)-\ell(\mu)}
    \frac{\operatorname{wt}(B_{\lambda,\mu})}{z_\mu}\,p_\mu,
    \qquad\text{where}\qquad
    z_\mu=\prod_{r\geq 1}r^{m_r(\mu)}m_r(\mu)! .
\]

\begin{lemma}\label{lemma1}
    For any labelled multiset $A$, we have
    \begin{equation}\label{eqlemma1}
        \wtF_A
        =
        \sum_{\pi\in\Pi(A)}
        \prod_{B\in\pi}(|B|-1)!\,F_{S_B},
    \end{equation}
    where $\Pi(A)$ denotes the set of set partitions of $A$, and the
    product runs over the blocks $B$ of $\pi$. Informally, $\pi$ groups the
    elements of $A$ into blocks, each block $B$ contributing a single part
    $S_B$ to a coarser partition, and the weight $(|B|-1)!$ counts the cyclic
    orderings of the elements of $B$.
\end{lemma}

\begin{proof}
    The case $A=\emptyset$ is immediate.  Assume $A\neq\emptyset$, and put
    $\lambda=\mu(A)$ and $n=S_A$.  For $\pi\in\Pi(A)$,
    let $\operatorname{sh}(\pi)$ be the partition obtained by rearranging the
    block sums $S_B$, $B\in\pi$, in weakly decreasing order.  By the
    E\u{g}ecio\u{g}lu--Remmel formula,
    \begin{equation}\label{eqERexp}
        \wtF_A
        =
        (-1)^{n-\ell(\lambda)}\left(\prod_{i\geq 1}m_i(A)!\right)m_\lambda
        =
        \sum_{\mu\vdash n}
        (-1)^{n-\ell(\mu)}
        \left(\prod_{i\geq 1}m_i(A)!\right)
        \frac{\operatorname{wt}(B_{\lambda,\mu})}{z_\mu}\,p_\mu .
    \end{equation}
    Thus it remains only to translate the coefficient into labelled-multiset
    notation.  We claim that, for each $\mu\vdash n$,
    \[
        \left(\prod_{i\geq 1}m_i(A)!\right)
        \operatorname{wt}(B_{\lambda,\mu})
        =
        z_\mu
        \sum_{\substack{\pi\in\Pi(A)\\ \operatorname{sh}(\pi)=\mu}}
        \prod_{B\in\pi}(|B|-1)! .
    \]
    The factor $\prod_i m_i(A)!$ is the number of ways to label the bricks by
    the distinct elements of $A$; hence the left-hand side is the sum of
    $\operatorname{wt}(T)$ over all labelled brick tabloids $T$ of shape
    $\mu$.  Such a labelled brick tabloid may be built by first choosing a set
    partition $\pi$ of $A$, whose blocks are the sets of bricks lying in the
    rows, then assigning the blocks with sum $r$ to the $m_r(\mu)$ rows of
    length $r$, and finally linearly ordering the bricks inside each row.  For a
    fixed block $B$, the sum of the rightmost-brick lengths over all linear
    orders of $B$ is
    \[
        \sum_{\text{linear orders of }B}\text{(length of the last brick)}
        =
        S_B(|B|-1)!,
    \]
    because each element of $B$ appears last in exactly $(|B|-1)!$ orders.
    Hence a fixed $\pi$ with $\operatorname{sh}(\pi)=\mu$ contributes total
    weight
    \[
        \left(\prod_{r\geq 1}m_r(\mu)!\right)
        \prod_{B\in\pi}S_B(|B|-1)!
        =
        \left(\prod_{r\geq 1}r^{m_r(\mu)}m_r(\mu)!\right)
        \prod_{B\in\pi}(|B|-1)!
        =
        z_\mu\prod_{B\in\pi}(|B|-1)! ,
    \]
    which proves the claim.  Substituting it into the E\u{g}ecio\u{g}lu--Remmel
    expansion~\eqref{eqERexp} gives
    \[
        \wtF_A
        =
        \sum_{\pi\in\Pi(A)}
        (-1)^{n-|\pi|}
        \left(\prod_{B\in\pi}(|B|-1)!\right)
        \prod_{B\in\pi}p_{S_B}.
    \]
    Absorbing the sign $(-1)^{n-|\pi|}=\prod_{B\in\pi}(-1)^{S_B-1}$ into the
    factors $p_{S_B}$ proves the lemma.
\end{proof}

\begin{lemma}\label{lemma2}
    Let $A$ be a labelled multiset, and let $c$ be a positive integer, regarded as a new labelled element. Then
    \begin{equation}\label{eqlemma2}
        \wtF_{A\cup\{c\}}
        =
        \sum_{T\subseteq A}
        |T|!\,\wtF_{A\setminus T}F_{c+S_T}.
    \end{equation}
\end{lemma}

\begin{proof}
    By Lemma~\ref{lemma1},
    \[
        \wtF_{A\cup\{c\}}
        =
        \sum_{\pi\in\Pi(A\cup\{c\})}
        \prod_{B\in\pi}(|B|-1)!\,F_{S_B}.
    \]
    We group the set partitions $\pi\in\Pi(A\cup\{c\})$ according to the
    block containing $c$. This block is $T\cup\{c\}$ for a unique subset
    $T\subseteq A$, and the remaining blocks form a set partition of
    $A\setminus T$. The block $T\cup\{c\}$ contributes the factor
    $(|T\cup\{c\}|-1)!\,F_{c+S_T}=|T|!\,F_{c+S_T}$, while the product over the
    remaining blocks, summed over all set partitions of $A\setminus T$,
    equals $\wtF_{A\setminus T}$ by Lemma~\ref{lemma1} applied to
    $A\setminus T$. Summing the resulting contributions
    $|T|!\,\wtF_{A\setminus T}F_{c+S_T}$ over $T\subseteq A$ gives
    \eqref{eqlemma2}.
\end{proof}

\begin{lemma}\label{lemma5}
    Let $B$ be a labelled multiset, let $r$ be a positive integer, and let
    $G$ be a symmetric function. For $U\subseteq B$, set
    \begin{equation}\label{eqdefKU}
        K_U
        :=
        \begin{cases}
            |U|!-(|U|-1)!\sum_{x\in U}q^x, & U\neq\emptyset,\\[1ex]
            1, & U=\emptyset.
        \end{cases}
    \end{equation}
    Then
    \begin{equation}\label{eqlemma5}
        \wtF_B\, C_r(G)
        =
        \sum_{U\subseteq B}
        K_U\, C_{r+S_U}(\wtF_{B\setminus U}G).
    \end{equation}
\end{lemma}

\begin{proof}
    By Lemma~\ref{lemma1}, we have
    \[
        \wtF_B
        =
        \sum_{\pi\in\Pi(B)}
        \prod_{D\in\pi}(|D|-1)!\,F_{S_D}.
    \]
    \begin{claim}
    For every set partition $\pi\in\Pi(B)$,
    \[
        \left(\prod_{D\in\pi}F_{S_D}\right) C_r(G)
        =
        \sum_{\mathcal U\subseteq\pi}
        \left(\prod_{D\in\mathcal U}(1-q^{S_D})\right)
        C_{r+S_{\mathcal U}}
        \left(
            \prod_{D\in\pi\setminus\mathcal U}F_{S_D}G
        \right),
    \]
    where $S_{\mathcal U}:=\sum_{D\in\mathcal U}S_D$.
    \end{claim}

    We argue by induction on the number of blocks $m=|\pi|$.
    By Lemma~\ref{lemma4}, moving a single factor past the operator gives
    $F_a\,C_i(H)=C_i(F_a H)+(1-q^a)\,C_{a+i}(H)$: the factor $F_a$ either
    passes inside the operator, leaving $C_i(F_a H)$, or is absorbed into it,
    raising the subscript by $a$ and producing the scalar $1-q^a$. When
    $m=0$ the product is empty and both sides equal $C_r(G)$. For the
    inductive step, fix a block $D_0\in\pi$ and set $\pi'=\pi\setminus\{D_0\}$.
    By the induction hypothesis,
    \[
        \left(\prod_{D\in\pi'}F_{S_D}\right) C_r(G)
        =
        \sum_{\mathcal V\subseteq\pi'}
        \left(\prod_{D\in\mathcal V}(1-q^{S_D})\right)
        C_{r+S_{\mathcal V}}
        \left(
            \prod_{D\in\pi'\setminus\mathcal V}F_{S_D}G
        \right).
    \]
    We multiply the induction hypothesis by $F_{S_{D_0}}$, obtaining
    $\bigl(\prod_{D\in\pi}F_{S_D}\bigr)C_r(G)$ on the left. On the right,
    Lemma~\ref{lemma4} (with $a=S_{D_0}$ and $i=r+S_{\mathcal V}$) splits the
    $\mathcal V$-summand in two: $F_{S_{D_0}}$ either passes inside the
    operator, producing the $\mathcal U=\mathcal V$ summand for $\pi$, or is
    absorbed, producing the $\mathcal U=\mathcal V\cup\{D_0\}$ summand. As
    $\mathcal V$ ranges over the subsets of $\pi'$, these summands exhaust the
    subsets $\mathcal U\subseteq\pi$, each exactly once, proving the claim.

    For $\mathcal U\subseteq\pi$, set $U=\bigcup_{D\in\mathcal U}D$, so that
    $S_{\mathcal U}=S_U$. Substituting the claim, we reorganize the resulting
    double sum over $\pi\in\Pi(B)$ and $\mathcal U\subseteq\pi$ by the set
    $U\subseteq B$: the absorbed blocks form a set partition $\rho\in\Pi(U)$,
    and the remaining blocks form a set partition $\sigma\in\Pi(B\setminus U)$.
    Hence
    \[
    \begin{aligned}
        \wtF_B C_r(G)
        &=
        \sum_{\pi\in\Pi(B)}
        \left(\prod_{D\in\pi}(|D|-1)!\right)
        \left(\prod_{D\in\pi}F_{S_D}\right)C_r(G) \\
        &=
        \sum_{\pi\in\Pi(B)}
        \left(\prod_{D\in\pi}(|D|-1)!\right)
        \sum_{\mathcal U\subseteq\pi}
        \left(\prod_{D\in\mathcal U}(1-q^{S_D})\right)
        C_{r+S_{\mathcal U}}
        \left(
            \prod_{D\in\pi\setminus\mathcal U}F_{S_D}G
        \right) \\
        &=
        \sum_{U\subseteq B}
        \sum_{\rho\in\Pi(U)}
        \sum_{\sigma\in\Pi(B\setminus U)}
        \left(
            \prod_{D\in\rho}(|D|-1)!(1-q^{S_D})
        \right)
        C_{r+S_U}
        \left(
            \prod_{E\in\sigma}(|E|-1)!\,F_{S_E}G
        \right) \\
        &=
        \sum_{U\subseteq B}
        \sum_{\rho\in\Pi(U)}
        \left(
            \prod_{D\in\rho}(|D|-1)!(1-q^{S_D})
        \right)
        C_{r+S_U}(\wtF_{B\setminus U}G) \\
        &=
        \sum_{U\subseteq B}
        L_U\, C_{r+S_U}(\wtF_{B\setminus U}G),
    \end{aligned}
    \]
    where
    \begin{equation}\label{eqdefLU}
        L_U
        :=
        \sum_{\rho\in\Pi(U)}
        \prod_{D\in\rho}(|D|-1)!(1-q^{S_D}).
    \end{equation}

    It suffices to show that $L_U=K_U$. Both sides equal $1$ when $U=\emptyset$, so we may assume $U\neq\emptyset$ and show that $\sum_{\rho\in\Pi(U)}\prod_{D\in\rho}(|D|-1)!(1-q^{S_D})=|U|!-(|U|-1)!\sum_{x\in U}q^x$. Since $\prod_{D\in\rho}(|D|-1)!$ is
    the number of permutations of $U$ whose cycles are the blocks of $\rho$,
    grouping the permutations $\sigma\in\mathfrak S_U$ by their cycle partition
    gives
    \[
        L_U
        =
        \sum_{\rho\in\Pi(U)}
        \Bigl(\prod_{D\in\rho}(|D|-1)!\Bigr)
        \prod_{D\in\rho}\bigl(1-q^{S_D}\bigr)
        =
        \sum_{\sigma\in\mathfrak S_U}
        \prod_{c\in\Cyc(\sigma)}\bigl(1-q^{S_c}\bigr),
    \]
    where $\Cyc(\sigma)$ is the set of cycles of $\sigma$ and $S_c$ is
    the sum of the elements of $c$. Expanding each product and grouping the
    terms by the $\sigma$-invariant union $W\subseteq U$ of the cycles that
    contribute a factor $-q^{S_c}$ gives
    \[
        L_U
        =
        \sum_{W\subseteq U}
        q^{S_W}\,|U\setminus W|!
        \sum_{\tau\in\mathfrak S_W}(-1)^{|\Cyc(\tau)|}.
    \]
    Since the numbers of even and odd permutations of a set of size at least
    $2$ are equal, we have
    \[
        \sum_{\tau\in\mathfrak S_W}(-1)^{|\Cyc(\tau)|}
        =
        (-1)^{|W|}\sum_{\tau\in\mathfrak S_W}\operatorname{sgn}(\tau)
        =
        \begin{cases}
            1, & W=\emptyset,\\
            -1, & |W|=1,\\
            0, & |W|\ge 2.
        \end{cases}
    \]
    Hence $L_U=|U|!-(|U|-1)!\sum_{x\in U}q^x=K_U$, and \eqref{eqlemma5} follows.
\end{proof}

\begin{lemma}\label{lemma6}
    Let $B$ be a labelled multiset, and let $b$ be a positive integer, regarded as a new labelled element. Then
    \begin{equation}\label{eqlemma6}
        \wtF_B F_b
        =
        \wtF_{B\cup\{b\}}
        -
        \sum_{x\in B}
        \wtF_{(B\setminus\{x\})\cup\{x+b\}}.
    \end{equation}
\end{lemma}

\begin{proof}
    By definition,
    \[
        \wtF_BF_b
        =
        (-1)^{S_B-|B|}M_B\cdot (-1)^{b-1}p_b
        =
        (-1)^{S_B+b-|B|-1}M_Bp_b.
    \]
    We claim that
    \begin{equation}\label{eqMBpb}
        M_Bp_b
        =
        M_{B\cup\{b\}}
        +
        \sum_{x\in B}
        M_{(B\setminus\{x\})\cup\{x+b\}}.
    \end{equation}
    In the product $M_B p_b$ with $p_b=\sum_i x_i^b$, each term is obtained by multiplying a monomial of $M_B$ by some $x_i^b$. The index $i$ either does not occur in that monomial, or is the index of exactly one labelled element $x\in B$. The first case contributes $M_{B\cup\{b\}}$, while the second case contributes
    $M_{(B\setminus\{x\})\cup\{x+b\}}$.

    Multiplying \eqref{eqMBpb} by $(-1)^{S_B+b-|B|-1}$, and using the definitions of $\wtF_{B\cup\{b\}}$ and $\wtF_{(B\setminus\{x\})\cup\{x+b\}}$, yields
    \[
        \wtF_BF_b
        =
        \wtF_{B\cup\{b\}}
        -
        \sum_{x\in B}
        \wtF_{(B\setminus\{x\})\cup\{x+b\}}.
    \]
    This proves the lemma.
\end{proof}

\section{Main results}
The main results of this section are Theorem~\ref{thm1}, which gives a
recursion for $\wtF_A$, and Theorem~\ref{thm2}, which derives
from it a positive $C$ expansion of $\wtF_A$. We then use this
expansion to prove our main results, Theorems~\ref{mainthm} and~\ref{maine}.
\subsection{Main theorems}
We regard $c$, as well as the integers $c-j$ and $b$ appearing in \eqref{eqdefTheta}, as new labelled elements when they are adjoined to a labelled multiset. For a labelled multiset $A$ of positive integers and an integer $c\geq 1$, we set
\begin{equation}\label{eqdefTheta}
    \Theta_{A,c}
    :=
    \wtF_{A\cup\{c\}}
    -
    \sum_{j=1}^{c-1}q^{j-1}C_j(\wtF_{A\cup\{c-j\}}).
\end{equation}

\begin{theorem}\label{thm1}
    Let $A$ be a labelled multiset of positive integers, and let $c\geq 1$ be an integer such that $x\geq c$ for every $x\in A$. Then
    \begin{equation}\label{eqthmTheta}
        \Theta_{A,c}
        =
        \sum_{T\subseteq A}
        \left(
            \sum_{b=1}^{S_T}
            \alpha_{T,b}\,
            C_{c+S_T-b}(\wtF_{(A\setminus T)\cup\{b\}})
            +
            \beta_T\,
            C_{c+S_T}(\wtF_{A\setminus T})
        \right),
    \end{equation}
    where
    \begin{equation}\label{eqdefalpha}
        \alpha_{T,b}
        =
        \begin{cases}
        (|T|-1)!
        \displaystyle\sum_{\substack{x\in T\\ b\leq x}}
        q^{\,(c+x-b-1) \bmod x},
        & T\neq\emptyset,\\[1.2em]
        0,
        & T=\emptyset,
        \end{cases}
    \end{equation}
    and
    \begin{equation}\label{eqdefbeta}
        \beta_T
        =
        (|T|+1)![c]_q
        +
        |T|!\sum_{x\in T}
        \left(q^c+q^{c+1}+\cdots+q^{x-1}\right).
    \end{equation}
    Here $(c+x-b-1)\bmod x$ denotes the least nonnegative residue of $c+x-b-1$ modulo $x$, so that all coefficients $\alpha_{T,b}$ and $\beta_T$ lie in $\N[q]$.
\end{theorem}

The proof of Theorem~\ref{thm1} is given in Section~4.2. Using
Theorem~\ref{thm1}, we now prove the following theorem.

\begin{theorem}\label{thm2}
    For any labelled multiset $A$ of positive integers, we have
    \[
        \wtF_A \in \sum_{\alpha\vDash S_A}\N[q]\, C_\alpha(1),
    \]
    where the sum ranges over all compositions $\alpha$ of $S_A$.
\end{theorem}

\begin{proof}
    Let $\mathcal{C}^+$ be the set of all $\N[q]$-linear combinations of the $C_\alpha(1)$, where $\alpha$ ranges over all compositions.
    Each $C_j$ is linear and satisfies $C_j(C_\alpha(1))=C_{(j,\alpha_1,\ldots,\alpha_m)}(1)$, so $\mathcal{C}^+$ is stable under every operator $C_j$ with $j\geq 1$. Moreover, $\wtF_A$ is homogeneous of degree $S_A$ while $C_\alpha(1)$ is homogeneous of degree $|\alpha|$, so any expansion of $\wtF_A$ in $\mathcal{C}^+$ involves only compositions of $S_A$.

    The empty multiset gives $\wtF_\emptyset=1=C_{\emptyset}(1)$, which settles
    this case.  For nonempty $A$, we argue by induction on the pair
    $(\ell(A),\min(A))$, ordered lexicographically, where $\min(A)$ is the
    smallest value in $A$.

    Let $c=\min(A)$, choose one labelled element of value $c$, and let $A'=A\setminus\{c\}$. Then
    $A=A'\cup\{c\}$ and $x\geq c$ for all $x\in A'$.

    By the definition~\eqref{eqdefTheta} of $\Theta_{A',c}$,
    \begin{equation}\label{eqthm2rec}
        \wtF_A
        =
        \Theta_{A',c}
        +
        \sum_{j=1}^{c-1}q^{j-1}C_j(\wtF_{A'\cup\{c-j\}}).
    \end{equation}

    We first show that $\Theta_{A',c}\in\mathcal{C}^+$. By Theorem~\ref{thm1},
    \[
        \Theta_{A',c}
        =
        \sum_{T\subseteq A'}
        \left(
            \sum_{b=1}^{S_T}
            \alpha_{T,b}\,
            C_{c+S_T-b}(\wtF_{(A'\setminus T)\cup\{b\}})
            +
            \beta_T\,
            C_{c+S_T}(\wtF_{A'\setminus T})
        \right).
    \]
    The coefficients $\alpha_{T,b}$ and $\beta_T$ lie in $\N[q]$, and the operator subscripts $c+S_T-b\geq c\geq 1$ and $c+S_T\geq 1$ are positive integers. Moreover, every $\wtF$ term occurring with a nonzero coefficient has length less than $\ell(A)$: for $\wtF_{A'\setminus T}$ we have $\ell(A'\setminus T)=\ell(A)-1-|T|<\ell(A)$, while a nonzero $\alpha_{T,b}$ forces $T\neq\emptyset$ and hence $\ell\bigl((A'\setminus T)\cup\{b\}\bigr)=\ell(A)-|T|<\ell(A)$. By the induction hypothesis these terms lie in $\mathcal{C}^+$, and since $\mathcal{C}^+$ is stable under each $C_j$, it follows that $\Theta_{A',c}\in\mathcal{C}^+$.

    It remains to consider the terms $q^{j-1}C_j(\wtF_{A'\cup\{c-j\}})$ in \eqref{eqthm2rec}, for $1\leq j\leq c-1$. The multiset $A'\cup\{c-j\}$ has the same length as $A$ and smaller minimum $c-j<c$, so the induction hypothesis gives $\wtF_{A'\cup\{c-j\}}\in\mathcal{C}^+$ and hence $q^{j-1}C_j(\wtF_{A'\cup\{c-j\}})\in\mathcal{C}^+$. Therefore every term on the right-hand side of \eqref{eqthm2rec} lies in $\mathcal{C}^+$, and the theorem follows.
\end{proof}

\begin{remark}\label{remexamples}
The proof of Theorem~\ref{thm2} is effective: iterating the recursion
in \eqref{eqdefTheta} and \eqref{eqthmTheta} produces an explicit expansion. For
instance, writing $C_{(a_1,\ldots,a_m)}(1)=C_{a_1}C_{a_2}\cdots C_{a_m}(1)$,
the recursion gives
\begin{align*}
    -m_{(2)} &= C_{(1,1)}(1) + (1+q)\,C_{(2)}(1),\\
    -m_{(2,1)} &= 2\,C_{(1,1,1)}(1) + (2+2q)\,C_{(1,2)}(1) + q\,C_{(2,1)}(1) + (2+q)\,C_{(3)}(1).
\end{align*}
For $\mu=(1^n)$ we have $\wtF_{(1^n)}=n!\,e_n$, and for all $n\leq 7$
the expansion produced by the recursion is
$\wtF_{(1^n)}=n!\sum_{\alpha\vDash n}C_\alpha(1)$, recovering the identity
\eqref{eqHMZen} of Haglund, Morse and Zabrocki \cite{HMZ12}.
\end{remark}

\begin{remark}\label{remverification}
Using SageMath, we verified Theorem~\ref{thm2} symbolically over $\Q(q)$
for all $272$ partitions $\mu$ with $|\mu|\leq 12$.  In each case, iterating
the recursion of Theorem~\ref{thm1} gives coefficients
$b_{\mu,\alpha}\in\N[q]$ with
$\sum_\alpha b_{\mu,\alpha}C_\alpha(1)=\wtF_\mu=(-1)^{|\mu|-\ell(\mu)}\wtm_\mu$.
\end{remark}

We are now ready to prove Theorem~\ref{mainthm}.
\begin{proof}[Proof of Theorem~\ref{mainthm}]
The idea is that $(-1)^{|\mu|-\ell(\mu)}\langle\nabla^r m_\mu,s_\lambda\rangle$
is at once an integer polynomial, by \eqref{nablamZ}, and a nonnegative rational
one, by Theorem~\ref{thm2} and Proposition~\ref{propnablaC}; being both, it lies
in $\N[q,t]$.  We now give the details.

Let $\mu\vdash n$, and choose any labelled multiset $A$ with
$\mu(A)=\mu$. By Theorem~\ref{thm2},
\[
    \wtF_A=\Bigl(\prod_{i\geq 1}m_i(\mu)!\Bigr)\,(-1)^{|\mu|-\ell(\mu)}m_\mu
    = \sum_{\alpha\vDash n} b_{\mu,\alpha}\, C_\alpha(1),
    \qquad b_{\mu,\alpha}\in \N[q].
\]
Dividing by $\prod_{i\geq 1}m_i(\mu)!$, the signed monomial symmetric
function has a nonnegative $C$ expansion:
\[
    (-1)^{|\mu|-\ell(\mu)}m_\mu
    = \sum_{\alpha\vDash n} a_{\mu,\alpha}\, C_\alpha(1),
    \qquad a_{\mu,\alpha}\in
    \Bigl(\prod_{i\geq 1}m_i(\mu)!\Bigr)^{-1}\N[q]
    \subset \Q_{\ge 0}[q].
\]
Applying $\nabla^r$ and taking the Hall inner product with $s_\lambda$, we obtain
\[
    (-1)^{|\mu|-\ell(\mu)}
    \langle \nabla^r m_\mu, s_\lambda\rangle
    =
    \sum_{\alpha\vDash n} a_{\mu,\alpha}
    \langle \nabla^r C_\alpha(1), s_\lambda\rangle .
\]
By Proposition~\ref{propnablaC}, each coefficient
$\langle \nabla^r C_\alpha(1), s_\lambda\rangle$ lies in
$\N[q,t]$. Hence the left-hand side belongs to
$\Q_{\ge0}[q,t]$.

On the other hand, by \eqref{nablamZ},
\[
    (-1)^{|\mu|-\ell(\mu)}
    \langle \nabla^r m_\mu, s_\lambda\rangle \in \Z[q,t].
\]
Both conditions hold, so
\[
    (-1)^{|\mu|-\ell(\mu)}
    \langle \nabla^r m_\mu, s_\lambda\rangle \in \N[q,t]. \qedhere
\]
\end{proof}

\begin{proof}[Proof of Theorem~\ref{maine}]
The proof is parallel to that of Theorem~\ref{mainthm}.  By Theorem~\ref{thm2},
\[
    (-1)^{|\mu|-\ell(\mu)}m_\mu
    =
    \sum_{\alpha\vDash n} a_{\mu,\alpha}\, C_\alpha(1),
    \qquad
    a_{\mu,\alpha}\in \Q_{\ge0}[q].
\]
Applying $\nabla^r$, substituting $q\mapsto q+1$, and taking the Hall inner
product with $\omega(m_\lambda)$, we obtain
\[
    (-1)^{|\mu|-\ell(\mu)}
    \left\langle (\nabla^r m_\mu)[X;q+1], \omega(m_\lambda)\right\rangle =
    \sum_{\alpha\vDash n}
    a_{\mu,\alpha}(q+1)
    \left\langle
        (\nabla^r C_\alpha(1))[X;q+1], \omega(m_\lambda)
    \right\rangle.
\]
Since $a_{\mu,\alpha}\in\Q_{\ge0}[q]$, we have $a_{\mu,\alpha}(q+1)\in\Q_{\ge0}[q]$,
so by Proposition~\ref{propnablaC} the left-hand side lies in $\Q_{\ge0}[q,t]$.
On the other hand, it lies in $\Z[q,t]$ by \eqref{nablamZe}.  Hence it lies in
$\N[q,t]$, as desired.
\end{proof}

\subsection{Proof of Theorem~\ref{thm1}}
Throughout this subsection, $A$ and $c$ are as in Theorem~\ref{thm1}; in
particular $x\geq c\geq 1$ for every $x\in A$, and $K_\emptyset=1$.

\begin{lemma}\label{lemmaThetaIntermediate}
    Under the same assumptions and notation, we have
    \begin{equation}\label{eqThetaInt}
    \begin{aligned}
        \Theta_{A,c}
        &=
        \sum_{T\subseteq A}
        \sum_{b=1}^{S_T}
        \left(
            \sum_{\substack{S\subseteq T\\ c+S_S>b}}
            |S|!\,q^{c+S_S-b-1}K_{T\setminus S}
        \right)
        C_{c+S_T-b}\left(\wtF_{A\setminus T}F_b\right) \\
        &\quad+
        \sum_{T\subseteq A}
        \left(
            \sum_{S\subseteq T}
            |S|!\,[c+S_S]_q K_{T\setminus S}
        \right)
        C_{c+S_T}(\wtF_{A\setminus T}).
    \end{aligned}
    \end{equation}
\end{lemma}

\begin{proof}
    By \eqref{eqdefTheta}, $\Theta_{A,c}=\wtF_{A\cup\{c\}}-\sum_{j=1}^{c-1}q^{j-1}C_j(\wtF_{A\cup\{c-j\}})$.
    We expand the two terms on the right-hand side in turn, starting with
    $\wtF_{A\cup\{c\}}$.  By Lemmas~\ref{lemma2} and~\ref{lemma3},
    \begin{align*}
        \wtF_{A\cup\{c\}}
        &=
        \sum_{S\subseteq A}|S|!\,\wtF_{A\setminus S}F_{c+S_S} \\
        &=
        \sum_{S\subseteq A}
        \sum_{b=1}^{c+S_S-1}
        |S|!\,q^{c+S_S-b-1}
        \wtF_{A\setminus S}C_{c+S_S-b}(F_b) +
        \sum_{S\subseteq A}
        |S|!\,[c+S_S]_q\,
        \wtF_{A\setminus S}C_{c+S_S}(1),
    \end{align*}
    where in the second line we applied \eqref{eqlemma3} with $n=c+S_S$ and
    reindexed the sum by $b=n-j$.
    We now apply Lemma~\ref{lemma5} with $B=A\setminus S$ to move each
    factor $\wtF_{A\setminus S}$ past the $C$ operator. If the subset
    absorbed in this step is $T\setminus S$ for some
    $S\subseteq T\subseteq A$, then the operator subscript increases by
    $S_{T\setminus S}$, becoming $c+S_T-b$ (respectively $c+S_T$), and a
    factor $K_{T\setminus S}$ is introduced. Thus
    \begin{equation}\label{eqfirstexp}
    \begin{aligned}
        \wtF_{A\cup\{c\}}
        &=
        \sum_{T\subseteq A}
        \sum_{b\geq 1}
        \left(
            \sum_{\substack{S\subseteq T\\ c+S_S>b}}
            |S|!\,q^{c+S_S-b-1}K_{T\setminus S}
        \right)
        C_{c+S_T-b}\left(\wtF_{A\setminus T}F_b\right) \\
        &\quad+
        \sum_{T\subseteq A}
        \left(
            \sum_{S\subseteq T}
            |S|!\,[c+S_S]_qK_{T\setminus S}
        \right)
        C_{c+S_T}(\wtF_{A\setminus T}).
    \end{aligned}
    \end{equation}

    The second sum of \eqref{eqfirstexp} is unchanged by the subtraction in
    \eqref{eqdefTheta} and gives the second sum of \eqref{eqThetaInt}.  We now
    expand the second term in \eqref{eqdefTheta}.  Again by Lemma~\ref{lemma2},
    \begin{equation}\label{eqsecondexp}
    \begin{aligned}
        \sum_{j=1}^{c-1}q^{j-1}C_j(\wtF_{A\cup\{c-j\}})
        &=
        \sum_{j=1}^{c-1}
        \sum_{S\subseteq A}
        q^{j-1}|S|!\,
        C_j\left(\wtF_{A\setminus S}F_{c-j+S_S}\right) \\
        &=
        \sum_{S\subseteq A}
        \sum_{b=S_S+1}^{c+S_S-1}
        |S|!\,q^{c+S_S-b-1}
        C_{c+S_S-b}
        \left(\wtF_{A\setminus S}F_b\right).
    \end{aligned}
    \end{equation}
    Here the second equality results from setting $b=c-j+S_S$ (equivalently
    $j=c+S_S-b$) to match the form of \eqref{eqfirstexp}.

    It remains to treat the first sum of \eqref{eqfirstexp}.  Fix
    $T\subseteq A$ and $b\geq1$.  The coefficient
    $\sum_{\substack{S\subseteq T\\ c+S_S>b}}|S|!\,q^{c+S_S-b-1}K_{T\setminus S}$
    of $C_{c+S_T-b}(\wtF_{A\setminus T}F_b)$ is nonzero only for
    $b\leq c+S_T-1$, and the summand indexed by $S=T$ is
    $|T|!\,q^{c+S_T-b-1}$, since $K_\emptyset=1$.  Let $S\subsetneq T$.
    Since $T\setminus S\neq\emptyset$ and every element of $A$ is at least
    $c$, we have $S_{T\setminus S}\geq c$.  Hence for $b\geq S_T+1$,
    \[
        c+S_S=c+S_T-S_{T\setminus S}\leq S_T<b,
    \]
    which contradicts the requirement $c+S_S>b$; the summand indexed by $S$ is
    then absent, and the coefficient reduces to the summand indexed by $S=T$.
    The part of the first sum of \eqref{eqfirstexp} with $b\geq S_T+1$ is
    therefore
    \[
        \sum_{T\subseteq A}\sum_{b=S_T+1}^{c+S_T-1}
        |T|!\,q^{c+S_T-b-1}\,
        C_{c+S_T-b}\left(\wtF_{A\setminus T}F_b\right),
    \]
    which is \eqref{eqsecondexp} after renaming the summation variable $S$ as
    $T$.  Subtracting \eqref{eqsecondexp} from the first sum of
    \eqref{eqfirstexp} removes its part with $b\geq S_T+1$, leaving the terms
    with $b\leq S_T$.

    Combining these surviving terms with the second sum of \eqref{eqfirstexp}
    gives
    \begin{align*}
        \Theta_{A,c}
        &=
        \sum_{T\subseteq A}
        \sum_{b=1}^{S_T}
        \left(
            \sum_{\substack{S\subseteq T\\ c+S_S>b}}
            |S|!\,q^{c+S_S-b-1}K_{T\setminus S}
        \right)
        C_{c+S_T-b}\left(\wtF_{A\setminus T}F_b\right) \\
        &\quad+
        \sum_{T\subseteq A}
        \left(
            \sum_{S\subseteq T}
            |S|!\,[c+S_S]_q K_{T\setminus S}
        \right)
        C_{c+S_T}(\wtF_{A\setminus T}).
        \qedhere
    \end{align*}
\end{proof}

\begin{lemma}\label{lemmaBeta}
    With the notation above, for every $T\subseteq A$ we have
    \[
        \sum_{S\subseteq T}
        |S|!\,[c+S_S]_q K_{T\setminus S}
        =
        (|T|+1)![c]_q
        +
        |T|!\sum_{x\in T}
        \left(q^c+q^{c+1}+\cdots+q^{x-1}\right)
        =\beta_T.
    \]
\end{lemma}

\begin{proof}
    Recall from~\eqref{eqdefKU} that
    \[
        K_{T\setminus S}=|T\setminus S|!-(|T\setminus S|-1)!\sum_{y\in T\setminus S}q^y .
    \]
    Let $|T|=k$. We compare the coefficients of $q^d$ on both sides.

    First suppose that $d<c$. Since $x\geq c$ for every $x\in T$, the definition of $K_{T\setminus S}$ shows that $K_{T\setminus S}$ contains no term of degree $d$ other than its constant term $|T\setminus S|!$. Hence the coefficient of $q^d$ on the left-hand side is
    \[
        \sum_{S\subseteq T}|S|!\,|T\setminus S|!
        =
        \sum_{i=0}^k \binom{k}{i}i!(k-i)!
        =
        (k+1)!.
    \]
    This agrees with the coefficient of $q^d$ on the right-hand side.

    Now suppose that $d\geq c$. We write
    \[
        \sum_{S\subseteq T}|S|!\,[c+S_S]_q K_{T\setminus S}
        =
        \sum_{S\subseteq T}|S|!\,[c+S_S]_q\,|T\setminus S|!
        -
        \sum_{\substack{S\subseteq T\\ y\in T\setminus S}}|S|!\,[c+S_S]_q(|T\setminus S|-1)!\,q^y .
    \]

    Since $[c+S_S]_q=1+q+\cdots+q^{c+S_S-1}$ has coefficient $1$ at $q^d$
    when $d<c+S_S$ and $0$ otherwise, the positive contribution to the
    coefficient of $q^d$ is
    \[
        \sum_{\substack{S\subseteq T\\ d<c+S_S}}|S|!\,|T\setminus S|!
        =
        \sum_{\substack{\emptyset\neq S\subseteq T\\ d<c+S_S}}\,
        \sum_{y\in S}(|S|-1)!\,|T\setminus S|!
        =
        \sum_{\substack{\emptyset\neq R\subseteq T,\ y\in R\\ d<c+S_R}}
        (|R|-1)!\,(|T|-|R|)! .
    \]
    The first equality drops the $S=\emptyset$ term, which vanishes as
    $d\geq c$, and expands $|S|!=\sum_{y\in S}(|S|-1)!$; the second relabels
    $S$ as $R$, putting the positive contribution in the same form as the
    negative one below so that the two can be subtracted.

    For the negative contribution, set $R=S\cup\{y\}$, where $y$ is the element chosen from $T\setminus S$. The term
    $q^y[c+S_S]_q$ contains $q^d$ if and only if $y\leq d<c+S_S+y=c+S_R$.

    Hence the negative contribution is
    \[
        \sum_{\substack{\emptyset\neq R\subseteq T,\ y\in R\\ y\leq d<c+S_R}}
        (|R|-1)!(|T|-|R|)!.
    \]

    After subtracting the negative contribution from the positive contribution, only the terms with $d<y$ remain. Moreover, if $d<y$, then $d<c+S_R$ automatically. Therefore the coefficient of $q^d$ on the left-hand side is
    \[
        \sum_{\substack{y\in T\\ d<y}}
        \sum_{\substack{R\subseteq T\\ y\in R}}
        (|R|-1)!(|T|-|R|)!.
    \]
    For each fixed $y\in T$, we have
    \begin{equation*}
        \sum_{\substack{R\subseteq T\\ y\in R}}
        (|R|-1)!(|T|-|R|)!
        =
        \sum_{W\subseteq T\setminus\{y\}}
        |W|!(k-1-|W|)!
        =
        \sum_{i=0}^{k-1}
        \binom{k-1}{i}i!(k-1-i)!
        =
        k!.
    \end{equation*}
    Thus, for $d\geq c$, the coefficient of $q^d$ on the left-hand side is $k!\,\#\{y\in T: d<y\}$.
    On the other hand, $[c]_q$ has degree less than $c$, so for $d\geq c$
    the coefficient of $q^d$ on the right-hand side is
    \[
        \left\langle q^d\right\rangle
        k!\sum_{x\in T}
        \left(q^c+q^{c+1}+\cdots+q^{x-1}\right)
        =
        k!\,\#\{x\in T: d<x\}.
    \]
    This matches the left-hand side, so the coefficients of $q^d$ agree for
    $d\geq c$ as well, proving the lemma.
\end{proof}

In the remainder of this section the coefficient of
$C_{c+S_T-b}(\wtF_{A\setminus T}F_b)$ in the first sum of \eqref{eqThetaInt}
appears repeatedly, so we let
\[
    L_{T,b}
    :=
    \sum_{\substack{S\subseteq T\\ c+S_S>b}}
    |S|!\,q^{c+S_S-b-1}K_{T\setminus S}.
\]

\begin{lemma}\label{lemmaAlphaCoeff}
    With the notation above, for every $T\subseteq A$ and every integer $b$, we have
    \begin{equation}\label{eqlemmaAlphaCoeff}
        L_{T,b}
        =
        \chi(b<c)\,|T|!\,q^{c-b-1}
        +
        \sum_{\emptyset\neq R\subseteq T}
        (|R|-1)!(|T|-|R|)!
        \sum_{\substack{y\in R\\ c+S_R-y\leq b<c+S_R}}
        q^{c+S_R-b-1}.
    \end{equation}
\end{lemma}

\begin{proof}
    Using the definition of $K_{T\setminus S}$, we obtain
    \[
        L_{T,b} =
        \sum_{\substack{S\subseteq T\\ c+S_S>b}} |S|!\,|T\setminus S|!\,q^{c+S_S-b-1}  -
        \sum_{\substack{S\subseteq T,\ c+S_S>b\\ y\in T\setminus S}} |S|!\,(|T\setminus S|-1)!\, q^{c+S_S+y-b-1}.
    \]

    We first rewrite the positive part. The contribution from $S=\emptyset$ is
    $\chi(b<c)\,|T|!\,q^{c-b-1}$.
    For $S\neq\emptyset$, we use $|S|! = \sum_{y\in S}(|S|-1)!$.

    Thus the positive part becomes
    \begin{equation}\label{eqPosPart}
        \chi(b<c)\,|T|!\,q^{c-b-1}
        +
        \sum_{\substack{\emptyset\neq R\subseteq T,\ y\in R\\ b<c+S_R}}
        (|R|-1)!(|T|-|R|)!\,
        q^{c+S_R-b-1}.
    \end{equation}

    We now rewrite the negative part. Set $R=S\cup\{y\}$,
    where $y\in T\setminus S$. Then
    \[
        S_R=S_S+y,
        \qquad
        |S|!= (|R|-1)!,
        \qquad
        (|T\setminus S|-1)!=(|T|-|R|)!.
    \]
    Moreover, the condition $c+S_S>b$ becomes
    $b<c+S_R-y$.
    Hence the negative part is
    \begin{equation}\label{eqNegPart}
        \sum_{\substack{\emptyset\neq R\subseteq T,\ y\in R\\ b<c+S_R-y}}
        (|R|-1)!(|T|-|R|)!\,
        q^{c+S_R-b-1}.
    \end{equation}

    After subtracting the negative part \eqref{eqNegPart} from the positive part
    \eqref{eqPosPart}, only the range
    $c+S_R-y\leq b<c+S_R$ remains for each fixed pair $(R,y)$. Therefore
    \[
        L_{T,b}
        =
        \chi(b<c)\,|T|!\,q^{c-b-1}
        +
        \sum_{\emptyset\neq R\subseteq T}
        (|R|-1)!(|T|-|R|)!
        \sum_{\substack{y\in R\\ c+S_R-y\leq b<c+S_R}}
        q^{c+S_R-b-1}.
        \qedhere
    \]
\end{proof}

Furthermore, we have the following lemma.

\begin{lemma}\label{lemmadfinal}
    With the notation above, we have
    \begin{equation}\label{eqlemmadfinal}
        \sum_{T\subseteq A}
        \sum_{b=1}^{S_T}
        L_{T,b}\,
        C_{c+S_T-b}
        \left(\wtF_{A\setminus T}F_b\right)
        =
        \sum_{T\subseteq A}
        \sum_{b=1}^{S_T}
        \alpha_{T,b}\,
        C_{c+S_T-b}
        \bigl(\wtF_{(A\setminus T)\cup\{b\}}\bigr) .
    \end{equation}
\end{lemma}

\begin{proof}
    Recall that
    \[
        L_{T,b}
        =
        \sum_{\substack{S\subseteq T\\ c+S_S>b}}
        |S|!\,q^{c+S_S-b-1}K_{T\setminus S}.
    \]
    By Lemma~\ref{lemma6} with $B=A\setminus T$, we have
    \begin{equation}\label{eqexpandFb}
        \wtF_{A\setminus T}F_b
        =
        \wtF_{(A\setminus T)\cup\{b\}}
        -
        \sum_{x\in A\setminus T}
        \wtF_{(A\setminus (T\cup\{x\}))\cup\{x+b\}}.
    \end{equation}

    Applying \eqref{eqexpandFb} to each summand on the left-hand side of
    \eqref{eqlemmadfinal} and using the linearity of $C_{c+S_T-b}$, we obtain
    \begin{equation}\label{eqLHSexpand}
    \begin{split}
        \sum_{T\subseteq A}\sum_{b=1}^{S_T}
        L_{T,b}\,C_{c+S_T-b}\bigl(\wtF_{A\setminus T}F_b\bigr)
        &=
        \sum_{T\subseteq A}\sum_{b=1}^{S_T}
        L_{T,b}\,C_{c+S_T-b}\bigl(\wtF_{(A\setminus T)\cup\{b\}}\bigr) \\
        &\quad-
        \sum_{T\subseteq A}\sum_{b=1}^{S_T}\sum_{x\in A\setminus T}
        L_{T,b}\,C_{c+S_T-b}\bigl(\wtF_{(A\setminus (T\cup\{x\}))\cup\{x+b\}}\bigr).
    \end{split}
    \end{equation}
    Reindex the triple sum by substituting $T\setminus\{x\}$ for $T$ and
    $b-x$ for $b$, keeping $x$: the condition $x\in A\setminus T$ becomes
    $x\in T$, and $1\leq b\leq S_T$ becomes
    $1\leq b-x\leq S_{T\setminus\{x\}}=S_T-x$, i.e., $x<b\leq S_T$.  Since
    \[
        c+S_{T\setminus\{x\}}-(b-x)=c+(S_T-x)-(b-x)=c+S_T-b
    \]
    and $\bigl(A\setminus((T\setminus\{x\})\cup\{x\})\bigr)\cup\{(b-x)+x\}=(A\setminus T)\cup\{b\}$,
    the triple sum equals
    \[
        \sum_{T\subseteq A}\sum_{b=1}^{S_T}
        \biggl(\,\sum_{\substack{x\in T\\ x<b}}L_{T\setminus\{x\},\,b-x}\biggr)
        C_{c+S_T-b}\bigl(\wtF_{(A\setminus T)\cup\{b\}}\bigr).
    \]
    Substituting this into \eqref{eqLHSexpand}, the left-hand side of
    \eqref{eqlemmadfinal} equals
    \begin{equation}\label{eqgammadef}
        \sum_{T\subseteq A}\sum_{b=1}^{S_T}
        \gamma_{T,b}\,C_{c+S_T-b}\bigl(\wtF_{(A\setminus T)\cup\{b\}}\bigr),
    \end{equation}
    where $\gamma_{T,b}:=L_{T,b}-\sum_{\substack{x\in T\\ x<b}}L_{T\setminus\{x\},\,b-x}$.
    Comparing \eqref{eqgammadef} with the right-hand side of
    \eqref{eqlemmadfinal}, it suffices to prove that $\gamma_{T,b}=\alpha_{T,b}$
    for all $T\subseteq A$ and $b\leq S_T$, where, by \eqref{eqdefalpha} and
    $T\neq\emptyset$,
    \[
        \alpha_{T,b}=(|T|-1)!\sum_{\substack{x\in T\\ b\leq x}}q^{\,(c+x-b-1)\bmod x}.
    \]

    By \eqref{eqlemmaAlphaCoeff},
    \begin{equation}\label{eqLexpand}
        L_{T,b}
        =
        \chi(b<c)\,|T|!\,q^{c-b-1}
        +
        \sum_{\emptyset\neq R\subseteq T}
        (|R|-1)!(|T|-|R|)!
        \sum_{\substack{y\in R\\ c+S_R-y\leq b<c+S_R}}
        q^{c+S_R-b-1},
    \end{equation}
    and
    \begin{multline}\label{eqLshift}
        L_{T\setminus\{x\},\,b-x}
        =
        \chi(b-x<c)\,(|T|-1)!\,q^{c+x-b-1} \\
        +
        \sum_{\emptyset\neq R'\subseteq T\setminus\{x\}}
        (|R'|-1)!(|T|-1-|R'|)!
        \sum_{\substack{y\in R'\\ c+S_{R'}-y\leq b-x<c+S_{R'}}}
        q^{c+S_{R'}-(b-x)-1}.
    \end{multline}
    Sum \eqref{eqLshift} over $x\in T$ with $x<b$.  Its first term, supported
    on $b<c+x$, gives
    \begin{equation}\label{eqfirsttype}
        \sum_{\substack{x\in T\\ x<b}}\chi(b-x<c)\,(|T|-1)!\,q^{c+x-b-1}
        =
        \sum_{\substack{x\in T\\ x<b<c+x}}(|T|-1)!\,q^{c+x-b-1},
    \end{equation}
    and its double sum, under $R=R'\cup\{x\}$ and then summation over
    $x\in R\setminus\{y\}$, gives
    \begin{align}
        &\sum_{\substack{x\in T\\ x<b}}
        \sum_{\emptyset\neq R'\subseteq T\setminus\{x\}}
        (|R'|-1)!(|T|-1-|R'|)!
        \sum_{\substack{y\in R'\\ c+S_{R'}-y\leq b-x<c+S_{R'}}}
        q^{c+S_{R'}-(b-x)-1} \notag\\
        &\qquad\overset{R=R'\cup\{x\}}{=}
        \sum_{\substack{R\subseteq T,\ |R|>1\\ y\in R,\ x\in R\setminus\{y\}\\ c+S_R-y\leq b<c+S_R}}
        (|R|-2)!(|T|-|R|)!\,q^{c+S_R-b-1} \notag\\
        &\qquad\overset{\phantom{R=R'\cup\{x\}}}{=}
        \sum_{\substack{R\subseteq T,\ |R|>1\\ y\in R\\ c+S_R-y\leq b<c+S_R}}
        (|R|-1)!(|T|-|R|)!\,q^{c+S_R-b-1}. \label{eqsecondtype}
    \end{align}
    The right-hand side of \eqref{eqsecondtype} is the $|R|>1$ part of
    \eqref{eqLexpand}, whose $|R|=1$ part ($R=\{x\}$, $S_R=x$) is
    $\sum_{x\in T,\ c\leq b<c+x}(|T|-1)!\,q^{c+x-b-1}$.  Hence
    \begin{align}
        \gamma_{T,b}
        &=
        L_{T,b}-\sum_{\substack{x\in T\\ x<b}}L_{T\setminus\{x\},\,b-x} \notag\\
        &=
        \chi(b<c)\,|T|!\,q^{c-b-1}
        +
        \sum_{\substack{x\in T\\ c\leq b<c+x}}(|T|-1)!\,q^{c+x-b-1}
        -
        \sum_{\substack{x\in T\\ x<b<c+x}}(|T|-1)!\,q^{c+x-b-1} \notag\\
        &=
        \chi(b<c)\,(|T|-1)!\sum_{x\in T}q^{c-b-1}
        +
        \sum_{\substack{x\in T\\ c\leq b\leq x}}(|T|-1)!\,q^{c+x-b-1}. \label{eqgammacollect}
    \end{align}

    It remains to rewrite the exponents as residues modulo $x$.  If $b<c$,
    then $b\leq x$ for every $x\in T$ since $x\geq c$; from
    $0\leq c-b-1<c\leq x$ and $c-b-1\equiv c+x-b-1\pmod{x}$ we obtain
    $c-b-1=(c+x-b-1)\bmod x$.  If $c\leq b\leq x$, then
    $0\leq c+x-b-1<x$, so $c+x-b-1=(c+x-b-1)\bmod x$.  In either case the
    exponent in \eqref{eqgammacollect} equals $(c+x-b-1)\bmod x$, and therefore
    \[
        \gamma_{T,b}
        =
        (|T|-1)!
        \sum_{\substack{x\in T\\ b\leq x}}
        q^{\,(c+x-b-1)\bmod x}
        =
        \alpha_{T,b}.
        \qedhere
    \]
\end{proof}

We are now ready to prove Theorem~\ref{thm1}.

\begin{proof}[Proof of Theorem~\ref{thm1}]
By Lemma~\ref{lemmaThetaIntermediate},
\begin{equation}\label{eqThetaLB}
\Theta_{A,c}
=
\sum_{T\subseteq A}\sum_{b=1}^{S_T}
L_{T,b}\,
C_{c+S_T-b}\bigl(\wtF_{A\setminus T}F_b\bigr)
+
\sum_{T\subseteq A}
\biggl(
    \sum_{S\subseteq T}
    |S|!\,[c+S_S]_q K_{T\setminus S}
\biggr)
C_{c+S_T}(\wtF_{A\setminus T}).
\end{equation}
By Lemma~\ref{lemmaBeta}, $\sum_{S\subseteq T}|S|!\,[c+S_S]_q K_{T\setminus S}=\beta_T$, and by
Lemma~\ref{lemmadfinal},
\begin{equation}\label{eqLalpha}
\sum_{T\subseteq A}\sum_{b=1}^{S_T}
L_{T,b}\,
C_{c+S_T-b}\bigl(\wtF_{A\setminus T}F_b\bigr)
=
\sum_{T\subseteq A}\sum_{b=1}^{S_T}
\alpha_{T,b}\,
C_{c+S_T-b}(\wtF_{(A\setminus T)\cup\{b\}}).
\end{equation}
This proves Theorem~\ref{thm1}.
\end{proof}

\section{Concluding remarks}
In this paper, we resolve Conjecture~\ref{conj:BGHT-IV} and prove its
higher-power analogue.

\subsection{An extension to the Schiffmann algebra}
We can extend the Schur positivity statement to the Schiffmann algebra
setting, as well as the $e$-positivity argument. Let $M=(1-q)(1-t)$. For
coprime positive integers $r,k$, we write $f[-MX^{r,k}]\cdot 1$ for the
action of the Schiffmann algebra on symmetric functions, introduced by Burban and Schiffmann
\cite{schiffmann}. Here we follow the notation of Blasiak, Haiman, Morse, Pun and Seelinger
\cite{BlasiakHaimanMorsePunSeelinger23b, BlasiakHaimanMorsePunSeelinger24, BlasiakHaimanMorsePunSeelinger25}. In particular,
$f[-MX^{r,1}]\cdot 1 = \nabla^r f$.

\begin{theorem}\label{thm:schiffmann}
Let $r,k$ be coprime positive integers and $\mu\vdash n, \lambda \vdash kn$. Then we have
\begin{equation}\label{schiffmanns}
    (-1)^{|\mu|-\ell(\mu)}\left\langle
        m_\mu\bigl[-MX^{r,k}\bigr]\cdot 1,
        s_\lambda
    \right\rangle
    \in \N[q,t].
\end{equation}
Moreover, after the substitution $q\mapsto q+1$, we have
\begin{equation}\label{schiffmanne}
    (-1)^{|\mu|-\ell(\mu)}
    \left\langle
        \left( m_\mu\bigl[-MX^{r,k}\bigr]\cdot 1 \right)[X;q+1],
        \omega(m_\lambda)
    \right\rangle
    \in \N[q,t].
\end{equation}
\end{theorem}

\begin{proof}
By Theorem~\ref{thm2}, $(-1)^{|\mu|-\ell(\mu)} m_\mu$ has a $C$ expansion
\[
    (-1)^{|\mu|-\ell(\mu)} m_\mu=\sum_{\alpha\vDash n} a_{\mu,\alpha}\,C_\alpha(1),
    \qquad
    a_{\mu,\alpha}\in\Q_{\ge0}[q].
\]
Via $|\mu|=|\alpha|=n$, applying the Schiffmann algebra action gives
\[
    (-1)^{|\mu|-\ell(\mu)} m_\mu[-MX^{r,k}]\cdot 1
    =\sum_{\alpha\vDash n} a_{\mu,\alpha}\, C_\alpha(1)[-MX^{r,k}]\cdot 1.
\]
By Mellit's compositional $(km,kn)$-shuffle theorem \cite{Mel21}, each
$C_\alpha(1)[-MX^{r,k}]\cdot 1$ is a nonnegative $\N[q,t]$-linear
combination of column LLT polynomials. Since LLT polynomials are Schur
positive, we have
\[
    (-1)^{|\mu|-\ell(\mu)} \left\langle m_\mu[-MX^{r,k}]\cdot 1,s_\lambda\right\rangle
    \in\Q_{\ge0}[q,t].
\]
Since $\langle s_\mu[-MX^{r,k}]\cdot 1, s_\lambda\rangle \in \N[q,t]$ by Blasiak, Haiman, Morse, Pun and Seelinger
\cite{BlasiakHaimanMorsePunSeelinger25} and the integrality of the inverse
Kostka matrix established by E\u{g}ecio\u{g}lu and Remmel \cite{ER90}, we have
\[
    (-1)^{|\mu|-\ell(\mu)} \left\langle m_\mu[-MX^{r,k}]\cdot 1,s_\lambda\right\rangle
    \in\Z[q,t].
\]
Hence we obtain \eqref{schiffmanns}.

For the $e$-positivity statement, we use the same expansion after the
substitution $q\mapsto q+1$. The shifted $e$-positivity of the corresponding
column LLT polynomials gives
\[
    \left\langle \left(C_\alpha(1)[-MX^{r,k}]\cdot 1\right)[X;q+1], \omega(m_\lambda)\right\rangle
    \in \N[q,t].
\]
Hence
\[
    (-1)^{|\mu|-\ell(\mu)}\left\langle \left(m_\mu[-MX^{r,k}]\cdot 1\right)[X;q+1], \omega(m_\lambda)\right\rangle
    \in \Q_{\ge0}[q,t].
\]
Since \eqref{schiffmanns} holds, the substitution $q\mapsto q+1$
preserves $\Z[q,t]$, and the transition matrix from the Schur basis
to the elementary basis has integer entries, we obtain
\[
    (-1)^{|\mu|-\ell(\mu)}
    \left\langle
        \left(m_\mu[-MX^{r,k}]\cdot 1\right)[X;q+1],
        \omega(m_\lambda)
    \right\rangle
    \in\Z[q,t].
\]
Thus we obtain \eqref{schiffmanne}.
\end{proof}

\subsection{Conjectures}
Nevertheless, our proof expresses
$(-1)^{|\mu|-\ell(\mu)}\nabla^r m_\mu$ only as a nonnegative rational combination
of column LLT polynomials, not as a nonnegative integral one; in particular,
it does not yield a direct parking function formula. Two natural questions arise.

First, Theorem~\ref{thm2} gives a $C$ expansion of $\wtF_\mu$, that is,
of $F_\mu=(-1)^{|\mu|-\ell(\mu)}m_\mu$ multiplied by $\prod_i m_i(\mu)!$,
with coefficients in $\N[q]$.  It is natural to ask whether
$F_\mu$ itself admits such an expansion.  The expansion produced by our
recursion, after dividing by $\prod_i m_i(\mu)!$, can have non-integral
coefficients; the smallest example is $\mu=(2,2)$, where the coefficient
of $C_{(1,2,1)}(1)$ equals $q/2$.  However, since the functions
$C_\alpha(1)$, $\alpha\vDash n$, do not form a basis, the $C$ expansion
of a symmetric function need not be unique, and we make the following
conjecture.

\begin{conjecture}\label{conjCpositive}
    For every partition $\mu$, the symmetric function
    $F_\mu=(-1)^{|\mu|-\ell(\mu)}m_\mu$ admits a $C$ expansion with
    coefficients in $\N[q]$.
\end{conjecture}

This holds whenever $\mu$ has distinct parts, since then
$\prod_i m_i(\mu)!=1$ and the expansion of Theorem~\ref{thm2} already has
coefficients in $\N[q]$.  For partitions with repeated parts we have
verified Conjecture~\ref{conjCpositive} using SageMath for all $\mu$ with
$|\mu|\leq 12$; for example,
\begin{multline*}
    m_{(2,2)}
    =
    C_{(1,1,1,1)}(1)
    +(q^3+q^2+q+1)\,C_{(1,1,2)}(1)
    +(q+1)\,C_{(1,3)}(1)\\
    +(q^2+2q)\,C_{(2,2)}(1)
    +(q+1)\,C_{(4)}(1),
\end{multline*}
and
\[
\begin{aligned}
m_{(3,3)}={}&
C_{(1,1,1,1,1,1)}(1)+C_{(1,1,1,1,2)}(1)+C_{(1,1,1,2,1)}(1)+C_{(1,1,1,3)}(1)\\
&+C_{(1,1,2,1,1)}(1)+C_{(1,1,2,2)}(1)+C_{(1,1,3,1)}(1)+C_{(1,1,4)}(1)\\
&+(1+q^2)\,C_{(1,2,1,1,1)}(1)+(1+q^3+q^4)\,C_{(1,2,1,2)}(1)+(1+q)\,C_{(1,2,2,1)}(1)\\
&+(1+q^3+q^4)\,C_{(1,2,3)}(1)+C_{(1,3,1,1)}(1)+(1+q)\,C_{(1,3,2)}(1)+C_{(1,4,1)}(1)\\
&+(1+q^2)\,C_{(1,5)}(1)+C_{(2,1,1,1,1)}(1)+C_{(2,1,1,2)}(1)+C_{(2,1,2,1)}(1)\\
&+C_{(2,1,3)}(1)+C_{(2,2,1,1)}(1)+C_{(2,2,2)}(1)+C_{(2,3,1)}(1)+C_{(2,4)}(1)\\
&+(1+q+q^2)\,C_{(3,1,1,1)}(1)+(1+q+q^2+2q^3)\,C_{(3,1,2)}(1)\\
&+(1+q+q^3)\,C_{(3,2,1)}(1)+(1+q+2q^2+q^3+q^4)\,C_{(3,3)}(1)\\
&+C_{(4,1,1)}(1)+(1+q^2)\,C_{(4,2)}(1)+C_{(5,1)}(1)+(1+q+q^2)\,C_{(6)}(1).
\end{aligned}
\]

By Mellit's compositional $(km,kn)$-shuffle theorem \cite{Mel21},
$\nabla^r C_\alpha(1)$ is a nonnegative integer combination of column LLT
polynomials. Conjecture~\ref{conjCpositive} would therefore imply that
$(-1)^{|\mu|-\ell(\mu)}\nabla^r m_\mu$ is also such a combination, refining
Theorem~\ref{mainthm} from Schur positivity to LLT positivity.

Second, one would like a combinatorial formula for
$(-1)^{|\mu|-\ell(\mu)}\nabla^r m_\mu$, ideally in terms of parking
functions.  A natural route is to find a $C$ expansion of $F_\mu$ with
coefficients in $\N[q]$ whose coefficients admit a combinatorial
description.  The compositional $(km,kn)$-shuffle theorem of Mellit \cite{Mel21} gives a
parking function formula for each $\nabla^r C_\alpha(1)$.  Such a
description would then yield the desired formula, extending those of Sergel
\cite{Ser18} for hooks and of Qu and Xin \cite{QX25} for
$\mu=(2^k,1^\ell)$.

\providecommand{\bysame}{\leavevmode\hbox to3em{\hrulefill}\thinspace}

\end{document}